\documentclass[a4paper,reqno,10pt]{amsart}

\usepackage[
  colorlinks=true,
  linkcolor=blue,
  citecolor=blue,
  urlcolor=blue
]{hyperref}
\usepackage{a4wide}
\usepackage{amssymb,amsmath,amsthm,amsfonts,mathtools}
\usepackage{enumitem}
\setlist[enumerate,1]{label={\upshape(\arabic*)}}
\setlist[enumerate,2]{label={\upshape(\alph*)}}
\usepackage{tikz-cd}
\usepackage{verbatim}

\newtheorem{theorem}{Theorem}[section]
\newtheorem{proposition}[theorem]{Proposition}
\newtheorem{lemma}[theorem]{Lemma}
\newtheorem{corollary}[theorem]{Corollary}
\newtheorem{theoremi}{Theorem}

\theoremstyle{definition}
\newtheorem{definition}[theorem]{Definition}
\newtheorem{example}[theorem]{Example}
\newtheorem{remark}[theorem]{Remark}
\newtheorem{question}[theorem]{Question}
\newtheorem*{acknowledgments}{Acknowledgments}
\newtheorem*{aiuse}{Use of AI}
\newtheorem*{organization}{Organization}
\newtheorem*{conventions}{Conventions and notation}
\numberwithin{equation}{section}

\newcommand{\kfield}{k}
\newcommand{\C}{\mathbb C}
\newcommand{\End}{\operatorname{End}}
\newcommand{\Hom}{\operatorname{Hom}}
\newcommand{\Ext}{\operatorname{Ext}}
\newcommand{\Gr}{\operatorname{Gr}}
\newcommand{\Rep}{\operatorname{Rep}}
\newcommand{\GL}{\operatorname{GL}}
\DeclareMathOperator{\tors}{\mathsf{tors}}
\DeclareMathOperator{\modu}{\mathsf{mod}}

\newcommand{\T}{\mathcal T}
\newcommand{\Tgen}{\mathsf T}

\DeclareMathOperator{\Filt}{\mathsf{Filt}}
\DeclareMathOperator{\add}{\mathsf{add}}
\newcommand{\Sol}{\mathcal S}
\newcommand{\parts}{\mathcal P}
\newcommand{\bcount}{\mathrm b}

\title[Finiteness and growth of brick chain filtrations]
{Finiteness and growth of brick chain filtrations}
\author[H.\ Enomoto]{Haruhisa Enomoto}
\address{Parakeet Inc., Japan}
\email{haruhisa.enomoto.math@gmail.com}
\date{\today}

\begin{document}

\begin{abstract}
We study Ringel's brick chain filtrations over finite-dimensional algebras:
filtrations whose factors are filtered by copies of individual bricks,
ordered so that morphisms from earlier bricks to later ones vanish.
First, using submodule varieties, we prove that the largest submodule of a
fixed module belonging to a torsion class takes only finitely many values
as the class varies, answering Pav\'on's question.
We deduce finiteness of brick chain filtrations and the bound $2^{d^2}$
for modules of dimension $d$.
For $\tau$-tilting finite algebras, we bound their number by the multinomial
coefficient determined by simple composition multiplicities.
Finally, we construct families of bricks over the three-arrow Kronecker
algebra whose filtration counts grow exponentially in the square of
composition length. We prove this growth using a Littlewood--Richardson
formula for submodule counts and the hook-length formula.
In particular, the counts eventually exceed the factorial of composition
length. These results answer Ringel's two questions.
\end{abstract}

\maketitle
\begingroup
\footnotesize
\setcounter{tocdepth}{2}
\tableofcontents
\endgroup
\enlargethispage{3pt}

\section{Introduction}

A composition series expresses a module in terms of simple modules. Let
$A$ be a finite-dimensional algebra over a field $k$,
and consider finite-dimensional right $A$-modules.
Ringel's brick chain filtrations use a larger class of factors: modules
filtered by copies of a single brick. Here a \emph{brick} is a module
whose endomorphism ring is a division ring. If $B$ is a brick, write
$\Filt B$ for the modules admitting a finite filtration whose factors are
isomorphic to $B$; the zero module is included via its empty filtration.
A \emph{brick chain filtration} of $M$ is a chain of submodules
\[
 0=M_0\subsetneq M_1\subsetneq\cdots\subsetneq M_t=M
\]
for which there exist bricks $B_1,\ldots,B_t$ such that
\[
 M_i/M_{i-1}\in\Filt B_i,\qquad
 \Hom_A(B_i,B_j)=0\quad(i<j).
\]
Thus a factor may contain several copies of its brick, while the order
of the factors is constrained by Hom-vanishing.

Ringel works more generally with modules of finite length over artin
algebras. In that setting, he proves that every module admits such a filtration
\cite[Theorem~1.2]{RingelFiltrations}. A filtration $(M_i)_i$ is
\emph{torsional} if $M_{i-1}$ belongs to the smallest torsion class
containing $M_i$ for every $i$. With this definition, Ringel proves that a module of
composition length $m$ has at most $m!$ torsional brick chain
filtrations \cite[Theorem~3.2]{RingelFiltrations}.
He then asks the following questions about all brick chain filtrations.

\begin{question}[{\cite[Questions following Theorem~3.2]{RingelFiltrations}}]\label{ques:ringel}
\leavevmode
\begin{enumerate}
\item Can a module have infinitely many brick chain filtrations?
\item Is there a module of composition length $m$ with more than $m!$
brick chain filtrations?
\end{enumerate}
\end{question}

For finite-dimensional algebras, we answer Question~\ref{ques:ringel}(1)
negatively and Question~\ref{ques:ringel}(2) affirmatively.

Let $\bcount_A(M)$ denote the number of chains $(M_i)_{i=0}^t$ forming
brick chain filtrations of $M$, and set $\bcount_A(0)=1$ for the empty
chain. We count actual chains of submodules of $M$.
The ordered isomorphism classes
$([B_1],\ldots,[B_t])$ form the \emph{type} of the filtration; this type
is determined by the chain
\cite[Lemma~9.7]{RingelFiltrations}.
Our first main result gives a bound depending only on the dimension of $M$.

\begin{theoremi}[= Theorem~\ref{thm:uniform}]
Let $A$ be a finite-dimensional algebra over a field $k$, and let $M$
be an $A$-module of dimension $d=\dim_k M$. Then
\[
 \bcount_A(M)\leq 2^{d^2}.
\]
\end{theoremi}

For split algebras, meaning that $\End_A(S)\simeq k$ as $k$-algebras
for every simple $A$-module $S$, Corollary~\ref{cor:uniform-length}
gives the bound $2^{m^2}$ in composition length $m$.

For a torsion class $\T$, the \emph{torsion part} $t_\T(M)$ is the
largest submodule of $M$ belonging to $\T$. We prove that, for every
fixed module $M$, only finitely many submodules occur as $t_\T(M)$
when $\T$ ranges over the torsion classes
(Theorem~\ref{thm:solid-finite}). This answers Pav\'on's
Question~2.19 in \cite{Pavon} for finite-dimensional algebras.

This result also proves finiteness of brick chain filtrations:
every filtration term is a torsion part
(Lemma~\ref{lem:filtration-solid}). Torsion parts are precisely the
submodules $U$ satisfying $\Hom_A(U,M/U)=0$, and this vanishing makes
them isolated points of the varieties parametrizing submodules of
each dimension. A B\'ezout inequality applied to the defining equations
then gives the estimate in Theorem~A.

For $\tau$-tilting finite algebras, every module of composition length $m$
has at most $m!$ brick chain filtrations.
Recall from \cite{DIJ} that an algebra is $\tau$-tilting finite if it has
only finitely many isomorphism classes of basic $\tau$-tilting modules
\cite{AIR}. This holds if and only
if the algebra has only finitely many torsion classes
\cite[Corollary~2.9(4) and Theorem~3.8]{DIJ}.
In particular, every representation-finite algebra has this property.

\begin{theoremi}[= Theorem~\ref{thm:multinomial}]
Let $A$ be a $\tau$-tilting finite algebra over a field $k$, and let
$S_1,\ldots,S_r$ represent its simple modules. Let $M$ have multiplicity
$d_i$ for $S_i$ and length $m=d_1+\cdots+d_r$. Then
\[
 \bcount_A(M)\leq\frac{m!}{d_1!\cdots d_r!}\leq m!.
\]
\end{theoremi}

For Theorem~B, we use the submodule polytope: the convex hull of the
simple composition multiplicity vectors of submodules of $M$.
Ringel's description of filtrations by chains of torsion classes and
the description of polytope edges by
Aoki--Higashitani--Iyama--Kase--Mizuno give an injection from brick
chain filtrations to increasing edge paths. We bound their number
by counting lattice paths. For nonsemisimple modules over
$\tau$-tilting finite algebras over an algebraically closed field,
Proposition~\ref{prop:nonsemisimple} gives the sharp bound
$m!/2+(m-1)!$. For $m\geq3$, it also shows that equality with $m!$
characterizes direct sums of pairwise nonisomorphic simple modules.

For general algebras, the exponent in Theorem~A has the correct
quadratic order. Let $Q_q$
denote the quiver with vertices $1,2$ and $q$ arrows $1\to2$.
Write $\Rep(Q_q,(a,b))$ for the affine space of representations with
vertex spaces $Me_1=k^a$ and $Me_2=k^b$, given by the $q$ arrow
maps $k^a\to k^b$ of the corresponding right $kQ_q$-module $M$.
Here and throughout, logarithms are natural.

\begin{theoremi}[= Theorem~\ref{thm:growth}]
Let $k$ be algebraically closed and $A=kQ_3$. There is a constant
$C>0$ such that, for every $n\geq1$,
a nonempty Zariski-open subset of $\Rep(Q_3,(3n,2n))$ consists of bricks
$M$ satisfying
\[
 \log\bcount_A(M)\geq (1-\log 2)n^2-Cn.
\]
\end{theoremi}

These modules have dimension and composition length $5n$. Consequently,
no uniform bound of the form $2^{o(d^2)}$ in dimension $d$, or
$2^{o(m^2)}$ in composition length $m$, is possible, even for bricks
over one fixed hereditary algebra.

For the five-arrow Kronecker algebra over an algebraically closed field,
we obtain an explicit example of length $244$.
A nonempty open subset of
$\Rep(Q_5,(167,77))$ consists of bricks of
composition length $167+77=244$ with more than $16\cdot244!$ two-step
brick chain filtrations $0\subsetneq U\subsetneq M$; see
Corollary~\ref{cor:244}. Thus Question~\ref{ques:ringel}(2) has an
affirmative answer even for a seven-dimensional hereditary algebra with
radical square zero.
Moreover, the filtrations counted here
have pairwise distinct types, so there are also more than $16\cdot244!$
types.

For these lower bounds, we count submodules $U$ with $U$ and $M/U$
bricks and $\Hom_A(U,M/U)=0$.
Schubert calculus on products of Grassmannians expresses
this submodule count through
Littlewood--Richardson coefficients. Their interpretation as
multiplicities for symmetric groups, together with the hook-length
formula for standard Young tableaux, gives the required estimates.

\begin{organization}
Section~\ref{sec:finiteness} identifies torsion parts with submodules $U$
satisfying $\Hom_A(U,M/U)=0$, proves finiteness of brick chain
filtrations, and answers Pav\'on's question.
Section~\ref{sec:upper} applies the B\'ezout inequality to affine charts
of the varieties of submodules to prove Theorem~A, and deduces the bound in composition
length for split algebras by Morita equivalence.
Section~\ref{sec:finite-type} relates filtrations to submodule polytopes,
proves Theorem~B and characterizes modules of composition length $m$
with $\bcount_A(M)=m!$.
Section~\ref{sec:kronecker} constructs bricks whose two-step filtrations
are counted by a sum of squares of Littlewood--Richardson coefficients.
Estimating this sum proves Theorem~C and gives the length-$244$ example
of Corollary~\ref{cor:244}. Its arithmetic is recorded in
Appendix~\ref{sec:certificate}.
\end{organization}

\begin{conventions}
Throughout, $k$ is a field. All algebras are associative, unital and
finite-dimensional over $k$, and all modules are finite-dimensional
right modules. Section~\ref{sec:kronecker} assumes that $k$ is
algebraically closed.
We write $\ell(M)$ for composition length and $\bcount_A(M)$ for the
number of brick chain filtrations of $M$, with $\bcount_A(0)=1$.
All logarithms are natural.
All subcategories are full and closed under isomorphism.
We write $\add X$ for the direct summands of finite direct sums of
copies of $X$.
Paths are multiplied in traversal order; an arrow $i\to j$ therefore acts
from $Me_i$ to $Me_j$. Matrices for quiver representations act on column
vectors.
\end{conventions}

\begin{acknowledgments}
The author thanks Sota Asai for helpful comments on the references.
\end{acknowledgments}

\begin{aiuse}
GPT-5.6-Sol and GPT-6-Astra were used for mathematical research and
manuscript writing under the author's direction. Opus~5 was used for
independent manuscript reviewing. The author determined the structure
and revised the exposition. The author is responsible for the result.
\end{aiuse}

\section{Finiteness of torsion parts and brick chain filtrations}\label{sec:finiteness}

We prove that a fixed module has only finitely many torsion parts as the
torsion class varies. Every term of a brick chain filtration is one of
these torsion parts, so finiteness of the filtrations follows.
Throughout this section, $M$ denotes a fixed $A$-module.

\subsection{Torsion parts of a fixed module}

We characterize the submodules that occur as torsion parts of a fixed
module and describe the torsion classes giving each part.

A torsion class in $\modu A$ is a subcategory closed under finite direct
sums, factor modules and extensions. For a torsion class $\T$, the sum
of the submodules of $M$ belonging to $\T$ is its largest such submodule,
denoted by $t_{\T}(M)$ and called the \emph{torsion part} of $M$ with
respect to $\T$. The quotient $M/t_\T(M)$ belongs to
\[
 \T^\perp=\{X\in\modu A\mid \Hom_A(T,X)=0\text{ for all }T\in\T\}.
\]
The pair $(\T,\T^\perp)$ is the associated torsion pair. Write
$\tors A$ for the set of torsion classes ordered by inclusion, and
$\Tgen(U)$ for the smallest torsion class containing a module $U$.
The exact sequence
\begin{equation}\label{eq:torsion-sequence}
 0\longrightarrow t_\T(M)\longrightarrow M
   \longrightarrow M/t_\T(M)\longrightarrow0,
 \qquad t_\T(M)\in\T,\quad M/t_\T(M)\in\T^\perp,
\end{equation}
is the \emph{torsion sequence} of $M$.

\begin{definition}\label{def:solid-submodule}
A submodule $U\subseteq M$ is \emph{solid in $M$} if
$\Hom_A(U,M/U)=0$. Set
\[
 \Sol_A(M)=\{U\subseteq M\mid \Hom_A(U,M/U)=0\}.
\]
\end{definition}

Ringel calls a filtration $(M_i)_i$ solid when
$\Hom_A(M_i/M_{i-1},M_j/M_{j-1})=0$ for every $i<j$
\cite[Section~10.1]{RingelFiltrations}. Hence the two-step filtration
$0\subsetneq U\subsetneq M$ is solid precisely when $U$ is solid in $M$
in the sense of Definition~\ref{def:solid-submodule}.
Pav\'on partitions $\tors A$ by equality of the torsion parts
$t_\T(M)$ and observes that every class of this partition is an interval
\cite[Section~1.3]{Pavon}. To describe these intervals explicitly, we
first record the elementary characterization of a torsion part used below.

\begin{lemma}\label{lem:torsion-part-characterization}
Let $\T\in\tors A$ and $U\subseteq M$. Then
\[
 t_\T(M)=U\quad\Longleftrightarrow\quad
 U\in\T\ \text{and}\ M/U\in\T^\perp.
\]
\end{lemma}

\begin{proof}
If $t_\T(M)=U$, the two conditions follow from the torsion sequence
\eqref{eq:torsion-sequence}. Conversely, suppose that $U\in\T$ and
$M/U\in\T^\perp$, and let $V\subseteq M$ belong to $\T$. The composite
$V\hookrightarrow M\twoheadrightarrow M/U$ is zero because
$V\in\T$ and $M/U\in\T^\perp$, so $V\subseteq U$.
Taking $V=t_\T(M)$ gives $t_\T(M)\subseteq U$. The reverse inclusion
follows because $t_\T(M)$ is the largest submodule of $M$ in $\T$.
\end{proof}

The following proposition identifies the possible torsion parts with
the solid submodules of $M$ and gives the intervals explicitly.

\begin{proposition}\label{prop:torsion-parts}
For $U\subseteq M$, the following conditions are equivalent:
\begin{enumerate}
\item $U\in\Sol_A(M)$.
\item $U=t_{\T}(M)$ for some $\T\in\tors A$.
\item $U=t_{\Tgen(U)}(M)$.
\end{enumerate}
If these conditions hold, then
\begin{equation}\label{eq:interval}
 \{\T\in\tors A\mid t_{\T}(M)=U\}
   =[\Tgen(U),{}^\perp(M/U)],
\end{equation}
where ${}^\perp(M/U)=\{X\mid\Hom_A(X,M/U)=0\}$.
\end{proposition}

\begin{proof}
(1) $\Rightarrow$ (3)
Left exactness of $\Hom_A(-,M/U)$ shows that ${}^\perp(M/U)$ is
closed under factor modules and extensions; closure under finite sums
follows from additivity. Since
$\Hom_A(U,M/U)=0$, this class contains $U$, and therefore contains
$\Tgen(U)$. Thus $\Hom_A(X,M/U)=0$ for every $X\in\Tgen(U)$,
which says that $M/U\in\Tgen(U)^\perp$. Since $U\in\Tgen(U)$,
Lemma~\ref{lem:torsion-part-characterization}
therefore gives (3).

(3) $\Rightarrow$ (2) Take $\T=\Tgen(U)$.

(2) $\Rightarrow$ (1)
The torsion sequence has $U\in\T$ and $M/U\in\T^\perp$, so
$\Hom_A(U,M/U)=0$.

Finally, Lemma~\ref{lem:torsion-part-characterization} shows that the
condition $t_\T(M)=U$ is equivalent to $U\in\T$ and
$M/U\in\T^\perp$. These conditions mean, respectively,
$\Tgen(U)\subseteq\T$ and $\T\subseteq{}^\perp(M/U)$, proving
\eqref{eq:interval}.
\end{proof}

\subsection{Hom-vanishing and isolated submodules}

We prove finiteness by showing that every solid submodule is an isolated
point in a variety of submodules. For
$0\leq s\leq\dim_\kfield M$, write $\Gr_A(s,M)$ for the algebraic subset of the
Grassmannian $\Gr(s,M)$ consisting of the $s$-dimensional submodules.
Here we use the sets of $k$-rational points with their Zariski topology.
The following coordinates describe the equations of $\Gr_A(s,M)$ and
the Hom spaces needed in the proof.

Fix a basis of $M$ and choose a coordinate decomposition $M=P\oplus Q$
of the underlying vector space, with $\dim P=s$. Relative to the ordered
decomposition $P\oplus Q$, matrices act on column vectors.
The subspaces whose projection to $P$ is an isomorphism form an affine
chart of $\Gr(s,M)$. They are the graphs
\[
 V_X=\{(p,Xp)\mid p\in P\},\qquad X\in\Hom_\kfield(P,Q).
\]
For $a\in A$, write the matrix of $m\mapsto ma$ as
$\left(\begin{smallmatrix}a_{11}&a_{12}\\a_{21}&a_{22}\end{smallmatrix}\right)$.
Its value on $(p,Xp)$ lies in the graph precisely when
\begin{equation}\label{eq:graph-invariance}
 a_{21}+a_{22}X=Xa_{11}+Xa_{12}X.
\end{equation}
It suffices to impose these equations for a basis of $A$. As $P$ ranges
over the coordinate $s$-planes, these charts cover $\Gr(s,M)$. Thus
$\Gr_A(s,M)$ is closed and has a finite cover by algebraic subsets of
affine spaces. The Hilbert basis theorem makes each chart Noetherian,
and hence $\Gr_A(s,M)$ is Noetherian.
The only quadratic term is $Xa_{12}X$, so the equations on each chart
have degree at most two. We will use this degree bound to count solid
submodules in Proposition~\ref{prop:solid-bound}.

For a fixed module $F$, we show that $\Hom_A(F,M/V)=0$ is an open
condition as $V$ varies.

\begin{lemma}\label{lem:hom-open}
For every $A$-module $F$ and every $0\leq s\leq\dim_k M$, the subset
\[
 \{V\in\Gr_A(s,M)\mid\Hom_A(F,M/V)=0\}
\]
is Zariski open.
\end{lemma}

\begin{proof}
On the chart above, $(p,q)\mapsto q-Xp$ identifies $M/V_X$ with $Q$.
Applying this map to $(0,q)a$ shows that the action of $a$ on the
quotient is $a_{22}-Xa_{12}$. If $F(a)$ is the action of $a$ on $F$,
a linear map $\phi:F\to Q$ is $A$-linear exactly when
\[
 (a_{22}-Xa_{12})\phi=\phi F(a)
\]
for every basis element $a$ of $A$. These are linear equations in
$\phi$, with coefficients polynomial in $X$. Their common kernel is
zero exactly when their coefficient matrix has full column rank.
This condition is open, since it is the nonvanishing of at least one
maximal minor. This proves openness on each chart of the cover above.
\end{proof}

Pav\'on asks whether, for a fixed module $M$, the set of submodules
$t_\T(M)$ is finite as $\T$ ranges over torsion classes
\cite[Question~2.19]{Pavon}. The following theorem answers this question
affirmatively for finite-dimensional algebras.

\begin{theorem}\label{thm:solid-finite}
Let $M\in\modu A$ and put $d=\dim_\kfield M$. Then the following hold.
\begin{enumerate}
\item\label{item:solid-isolated} For every $0\leq s\leq d$ and every
$U\in\Sol_A(M)$ of dimension $s$, the point $U$ is isolated in
$\Gr_A(s,M)$.
\item\label{item:solid-finite} The set $\Sol_A(M)$ is finite.
\item\label{item:solid-intervals} The fibers of
$\T\mapsto t_\T(M)$ partition $\tors A$ into finitely many intervals
indexed by $U\in\Sol_A(M)$.
\end{enumerate}
\end{theorem}

\begin{proof}
(1) Fix $U\in\Sol_A(M)$ of dimension $s$. By Lemma~\ref{lem:hom-open},
\[
 \Omega_U=\{V\in\Gr_A(s,M)\mid\Hom_A(U,M/V)=0\}
\]
is open and contains $U$ because $\Hom_A(U,M/U)=0$. For every
$V\in\Omega_U$,
the composite $U\hookrightarrow M\twoheadrightarrow M/V$ is zero.
Hence $U\subseteq V$, and equality follows from $\dim U=\dim V=s$.
We have proved $\Omega_U=\{U\}$, so $U$ is isolated in
$\Gr_A(s,M)$.

(2) The singleton $\{U\}$ is open by (1) and closed because $U$ is a
$k$-rational point. Any irreducible component containing $U$ must equal
$\{U\}$: otherwise $\{U\}$ and its complement would partition that
component into two nonempty closed subsets. There are only finitely many
irreducible components of a Noetherian space, and only finitely many possible
values of $s$. This proves finiteness of $\Sol_A(M)$.

(3) Proposition~\ref{prop:torsion-parts} gives the interval partition,
which is finite by (2).
\end{proof}

\begin{remark}
Tangent spaces give another proof of isolation. Consider the closed
subscheme of the Grassmannian over $k$ defined by the submodule equations;
its $k$-rational points form $\Gr_A(s,M)$. A first-order deformation of a
submodule $U$ is represented by a linear map $U\to M/U$; linearizing
the invariance equations says precisely that this map is $A$-linear.
Thus its Zariski tangent space at $U$ is $\Hom_A(U,M/U)$.
If $U$ is solid, this tangent space is zero. In the local ring at $U$,
the maximal ideal $\mathfrak m$ therefore satisfies
$\mathfrak m/\mathfrak m^2=0$. Nakayama's lemma gives $\mathfrak m=0$,
so $U$ is an isolated reduced point.
\end{remark}

\subsection{Finiteness of brick chain filtrations}

We next show that every term of a brick chain filtration is solid.
Finiteness of the possible terms will then imply finiteness of the
filtrations.

For a class $\mathcal C$ of modules, let $\Filt\mathcal C$ consist of
the modules admitting a finite filtration with factors isomorphic to
members of $\mathcal C$, including the zero module via the empty
filtration. We write $\Filt B$ when $\mathcal C=\{B\}$.

\begin{definition}[{\cite[Section~1.1]{RingelFiltrations}}]\label{def:bcf}
A \emph{brick chain} is a sequence of bricks $(B_1,\ldots,B_t)$ with
\[
 \Hom_A(B_i,B_j)=0\qquad(i<j).
\]
A \emph{brick chain filtration} of $M$
is a chain
\[
 0=M_0\subsetneq M_1\subsetneq\cdots\subsetneq M_t=M
\]
for which there is a brick chain $(B_1,\ldots,B_t)$ satisfying
\[
 M_i/M_{i-1}\in\Filt B_i\qquad(1\leq i\leq t).
\]
\end{definition}

For $M=0$, we use the additional convention that $t=0$ and that the
empty chain is a brick chain filtration. Following Ringel, the ordered
sequence $([B_1],\ldots,[B_t])$ is the \emph{type} of a filtration.
If a nonzero module belongs to both $\Filt B$ and $\Filt C$, then $B\simeq C$
\cite[Lemma~9.7]{RingelFiltrations}. Hence the type is determined by the
chain of submodules.

We also record the induced quotient filtration, which will be used to
count filtrations recursively in Theorem~\ref{thm:uniform}.

\begin{lemma}\label{lem:filtration-solid}
Every term of a brick chain filtration of $M$ belongs to $\Sol_A(M)$.
If its first nonzero term is $U$, the induced filtration of $M/U$ is a
brick chain filtration.
\end{lemma}

\begin{proof}
Let $(B_1,\ldots,B_t)$ be the type of the filtration. We have
$M_s\in\Filt\{B_1,\ldots,B_s\}$ and
$M/M_s\in\Filt\{B_{s+1},\ldots,B_t\}$.
Since Hom-vanishing is preserved under extensions in either variable,
the brick chain condition gives
$\Hom_A(M_s,M/M_s)=0$.
If $U=M_1$, the induced quotient filtration is
\[
 0=M_1/U\subsetneq M_2/U\subsetneq\cdots\subsetneq M_t/U=M/U.
\]
Its successive factors belong to
$\Filt B_2,\ldots,\Filt B_t$, proving the last assertion.
\end{proof}

\begin{corollary}\label{cor:bcf-finite}
Every module has finitely many brick chain filtrations. Only finitely
many isomorphism classes of bricks occur in their types.
\end{corollary}

\begin{proof}
By Lemma~\ref{lem:filtration-solid}, each filtration is a strict chain
in $\Sol_A(M)$, which is finite by
Theorem~\ref{thm:solid-finite}\ref{item:solid-finite}. There are only finitely many such
chains, which proves the first assertion. Each chain has finitely many
factors, and each nonzero factor belonging to some $\Filt B$ determines
$B$ up to isomorphism by \cite[Lemma~9.7]{RingelFiltrations}.
This proves the second assertion.
\end{proof}

This answers Question~\ref{ques:ringel}(1) negatively.

We compute the semisimple case to exhibit modules attaining the
factorial count in Question~\ref{ques:ringel}(2).

\begin{example}\label{ex:semisimple}
Let $S_1,\ldots,S_r$ be pairwise nonisomorphic simple modules and let
$M=\bigoplus_{i=1}^r S_i^{a_i}$, with all $a_i>0$.
A submodule $U$ has the form $\bigoplus_i U_i$ with
$U_i\subseteq S_i^{a_i}$ and $U_i\simeq S_i^{b_i}$.
If $0<b_i<a_i$, then both $U_i$ and $S_i^{a_i}/U_i$ contain $S_i$,
so $\Hom_A(U,M/U)\ne0$. If every $b_i$ is either $0$ or $a_i$,
the isomorphism classes of simple summands in $U$ and $M/U$ are
disjoint, so
$\Hom_A(U,M/U)=0$. Thus the solid submodules are exactly the sums
of complete isotypic summands.

A factor of the semisimple module $M$ that belongs to $\Filt B$ has $B$
as a subquotient. Hence $B$ is
semisimple; a semisimple brick is simple, since a decomposition into two
nonzero summands would give a nontrivial idempotent endomorphism.
The factor is therefore
isotypic. Every term in a brick chain filtration is solid, so it is a
sum of complete isotypic summands. It follows that each step adds
exactly one such summand. Each complete isotypic summand is added at
exactly one step, and pairwise nonisomorphic simple modules have no
nonzero morphisms between them. Hence every
ordering of the $r$ isotypic summands gives a brick chain filtration,
and these are all the filtrations. Their number is $r!$.
In particular, if every $a_i=1$, then $\ell(M)=r$ and the count is
$\ell(M)!$.
\end{example}

\section{A uniform bound for brick chain filtrations}\label{sec:upper}

We prove that a module of dimension $d$ has at most
$2^{d^2}$ brick chain filtrations, independently of the algebra.
We first bound the number of possible filtration terms and then count
filtrations by induction on dimension. For split algebras, Morita
equivalence gives the bound $2^{m^2}$ in composition length $m$.

We use the following bound for isolated common zeros of polynomial
equations. It is a special case of \cite[Theorem~5]{TaoBezout};
see also \cite[Theorem~1]{Heintz} for the general affine B\'ezout
inequality.

\begin{lemma}\label{lem:bezout}
Let $\kfield$ be algebraically closed, and let
$Z\subseteq\mathbb A^N_{\kfield}$ be the common zero set of finitely
many polynomials of degree at most $e$, where $e\geq1$. Then $Z$ has
at most $e^N$ isolated points.
\end{lemma}

\begin{proof}
Theorem~5 of \cite{TaoBezout} bounds the number of isolated points by
the product of the $N$ largest degrees for a list of at least $N$
nonconstant polynomials. After discarding zero polynomials, we can
repeat one polynomial to meet this requirement; each degree is at most
$e$. A nonzero constant equation gives the empty set, and an empty list
gives $\mathbb A^N_k$, so these cases also satisfy the bound.
\end{proof}

In \eqref{eq:graph-invariance}, the number of equations may depend
on $A$, but their degree and the number of chart coordinates depend
only on $\dim_\kfield M$. This gives a bound independent of the algebra.

\begin{proposition}\label{prop:solid-bound}
Let $d=\dim_{\kfield}M$. For $0\leq s\leq d$, one has
\[
 \#\{U\in\Sol_A(M)\mid\dim_{\kfield}U=s\}
 \leq \binom ds 2^{s(d-s)}.
\]
In particular,
\[
 \#\Sol_A(M)\leq\sum_{s=0}^d\binom ds 2^{s(d-s)}.
\]
\end{proposition}

\begin{proof}
Let $K$ be an algebraic closure of $k$, and write $A_K=A\otimes_k K$
and $X_K=X\otimes_k K$ for an $A$-module $X$. Scalar extension gives
\[
 \Hom_{A_K}(U_K,(M/U)_K)
   \simeq \Hom_A(U,M/U)\otimes_k K.
\]
Indeed, each Hom space is the kernel of the same finite system of
linear equations, before and after extending scalars. Thus $U_K$ is
solid in $M_K$ whenever $U$ is solid in $M$. Distinct submodules remain
distinct after scalar extension, and $\dim_K U_K=\dim_k U$.
It therefore suffices to prove the bound over the algebraically closed
field $K$.

Fix a basis of $M_K$. The Grassmannian $\Gr(s,M_K)$ has a cover by
$\binom ds$ affine charts, each with $s(d-s)$ coordinates.
On each chart, \eqref{eq:graph-invariance} gives equations of degree
at most two. Theorem~\ref{thm:solid-finite}\ref{item:solid-isolated} shows that every solid
submodule is isolated in $\Gr_{A_K}(s,M_K)$, and hence in each chart
containing it.
Lemma~\ref{lem:bezout} bounds their number
on each chart by $2^{s(d-s)}$. Summing over the charts can only
overcount points, so it gives the required upper bound.
\end{proof}

Counting filtrations by their first nonzero term gives the following bound.

\begin{theorem}\label{thm:uniform}
Let $A$ be a finite-dimensional algebra over a field $k$, and let $M$
be an $A$-module of dimension $d=\dim_k M$. Then
\[
 \bcount_A(M)\leq 2^{d^2}.
\]
\end{theorem}

\begin{proof}
We argue by induction on $d$. For $d=0$ the bound is our
convention $\bcount_A(0)=1$.
For $d\geq1$, a filtration is determined by its first nonzero term $U$
and the induced filtration of $M/U$. By
Lemma~\ref{lem:filtration-solid} and Proposition~\ref{prop:solid-bound},
the contribution of terms $U$ of dimension $s$ is at most
$\binom ds2^{s(d-s)}2^{(d-s)^2}$. We use
$\binom ds\leq d^s$ and $d\leq2^{d-1}$ for $d\geq1$; the latter
follows by induction from $d+1\leq2d$. Hence
\begin{align*}
 \bcount_A(M)
 &\leq\sum_{s=1}^d\binom ds 2^{s(d-s)}2^{(d-s)^2}\\
 &=2^{d^2}\sum_{s=1}^d\binom ds 2^{-ds}\\
 &\leq2^{d^2}\sum_{s=1}^d\left(\frac d{2^d}\right)^s
 \leq2^{d^2}\sum_{s=1}^d2^{-s}<2^{d^2}.
\end{align*}
Here the equality uses
$s(d-s)+(d-s)^2=d(d-s)=d^2-ds$.
\end{proof}

For split algebras we can express the bound in terms of composition
length. Recall that $A$ is \emph{split} if $\End_A(S)\simeq k$ as $k$-algebras
for every simple $A$-module $S$. Every finite-dimensional algebra over an
algebraically closed field is split. A basic split algebra has one-dimensional
simple modules, so dimension and composition length agree for its modules.

\begin{corollary}\label{cor:uniform-length}
Let $A$ be a split finite-dimensional algebra over a field $k$, and let
$M$ be an $A$-module of composition length $m$. Then
\[
 \bcount_A(M)\leq 2^{m^2}.
\]
\end{corollary}

\begin{proof}
Choose a Morita equivalence $F:\modu A\to\modu A'$ with $A'$ basic. The exact
functor $F$ identifies subobject lattices and Hom spaces, sends bricks to
bricks, and satisfies $F(\Filt B)=\Filt(F(B))$. It therefore preserves the
number of brick chain filtrations and composition length. It also
identifies the endomorphism algebras of corresponding simple modules,
so $A'$ is split. Since $A'$ was chosen basic,
$\dim_k F(M)=\ell(F(M))=m$, and
Theorem~\ref{thm:uniform} applied to $F(M)$ gives the asserted bound.
\end{proof}

\section{Bounds for brick chain filtrations over \texorpdfstring{$\tau$}{tau}-tilting finite algebras}\label{sec:finite-type}

For $\tau$-tilting finite algebras, we bound the number of brick chain
filtrations by a multinomial coefficient in the simple composition
multiplicities, and hence by $\ell(M)!$. The proof uses the submodule
polytope: filtrations give increasing edge paths, which can be counted
by coordinate lattice paths. We then characterize the modules satisfying
$\bcount_A(M)=\ell(M)!$.
Recall that $\tors A$ is finite for every $\tau$-tilting finite algebra
$A$ \cite[Theorem~3.8]{DIJ}.

\subsection{Submodule polytopes and filtration steps}

The submodule polytope records the composition multiplicities of
submodules. We use the description of its vertices and edges in
\cite[Section~5.3]{AHIKM} to interpret the successive terms of a
brick chain filtration.

Choose representatives $S_1,\ldots,S_r$ of the simple $A$-modules.
For a module $X$, let $[X]\in\mathbb Z_{\geq0}^r$ be its vector of
simple composition multiplicities. Define
\[
 P(M)=\operatorname{conv}\{[U]\mid U\subseteq M\}
       \subseteq\mathbb R^r.
\]
This is an integral polytope, since there are only finitely many
possible vectors $[U]$. Baumann, Kamnitzer and Tingley
\cite[Section~1.3]{BKT} and Fei \cite{FeiComb} study the corresponding
polytope of dimension vectors of submodules. Put $\mathbf d=[M]=(d_1,\ldots,d_r)$. Then $P(M)$ lies in the
coordinate box $[0,\mathbf d]=\prod_i[0,d_i]$. The sum of the
coordinates has its unique minimum at $0$ and unique maximum at
$\mathbf d$, so both are vertices. We call an edge from $u$ to $v$
\emph{increasing} if $u\leq v$ coordinatewise and $u\ne v$.
An increasing edge path follows such edges from $0$ to $\mathbf d$.

A \emph{cover} $\T'\lessdot\T$ in $\tors A$ is a strict inclusion
with no torsion class strictly between its endpoints. The following
result is the part of the polytope description that we need.

\begin{theorem}[{\cite[Theorem~5.23(b),(c)]{AHIKM}}]\label{thm:polytope-torsion}
Let $A$ be $\tau$-tilting finite and let $M\in\modu A$.
\begin{enumerate}
\item\label{item:polytope-vertices}
The vertices of $P(M)$ are precisely the vectors $[t_\T(M)]$ for
$\T\in\tors A$.
\item\label{item:polytope-edges}
For every cover $\T'\lessdot\T$, either $t_{\T'}(M)=t_\T(M)$ or
$[t_{\T'}(M)],[t_\T(M)]$ are the endpoints of an increasing edge.
Every edge of $P(M)$ is obtained from such a cover.
\end{enumerate}
\end{theorem}

We next recall the brick labels of covers in $\tors A$ and the
filtrations associated with maximal chains. A brick chain
$(B_1,\ldots,B_t)$ is \emph{complete} if there do not exist a brick
$B$ and an index $0\leq i\leq t$ such that
$(B_1,\ldots,B_i,B,B_{i+1},\ldots,B_t)$ is a brick chain.
We regard brick chains up to isomorphism of their corresponding bricks.

\begin{proposition}[{\cite[Theorems~3.3 and~3.4]{DIRRT};
\cite[Theorem~A.3]{Keller}; \cite[Theorems~2.7 and~11.2]{RingelFiltrations}}]\label{prop:ringel-chains}
Let $A$ be $\tau$-tilting finite.
\begin{enumerate}
\item\label{item:ringel-label}
Every cover $\T'\lessdot\T$ has a unique brick $B\in\T$, up to
isomorphism, such that $\T'=\T\cap{}^\perp B$. This brick is called
its \emph{label}, and $\T$ is the torsion class generated by $\T'$ and $B$.
Every module $X\in\T$ has a submodule $X'\in\T'$ with
$X/X'\in\Filt B$.
\item\label{item:ringel-complete}
Every brick chain is a subsequence of a complete brick chain. Labeling
the successive covers in a maximal chain of $\tors A$, in increasing
order, gives a bijection from maximal chains in $\tors A$ to complete brick
chains, up to isomorphism of their bricks.
\item\label{item:ringel-unique}
For every complete brick chain and every module $M$, there is a unique
brick chain filtration of $M$ whose type is a subsequence of that chain.
\end{enumerate}
\end{proposition}

Part~\ref{item:ringel-label} follows from
\cite[Theorems~3.3(b) and~3.4(b)]{DIRRT} and is stated in this form in
\cite[Theorem~2.7]{RingelFiltrations}; see also
\cite[Theorems~1.2 and~1.3]{BCZ}, and \cite[Definition~2.14]{Asai} for
functorially finite torsion classes.
For part~\ref{item:ringel-complete}, there are only finitely many bricks
up to isomorphism by \cite[Theorem~4.2]{DIJ}, so inserting further bricks
must eventually give a complete chain. The bijection is proved by Demonet
in his appendix to Keller's survey \cite[Theorem~A.3]{Keller}.
Part~\ref{item:ringel-unique} is \cite[Theorem~11.2]{RingelFiltrations}.

The next lemma describes the filtration in
Proposition~\ref{prop:ringel-chains}\ref{item:ringel-unique}
by evaluating torsion parts.

\begin{lemma}\label{lem:evaluate-chain}
Let $A$ be $\tau$-tilting finite, let $M\in\modu A$, and let
$0=\T_0\lessdot\T_1\lessdot\cdots\lessdot\T_n=\modu A$ be a
maximal chain, with successive labels $B_1,\ldots,B_n$. After deleting
repeated terms, the chain
\[
 0=t_{\T_0}(M)\subseteq t_{\T_1}(M)\subseteq\cdots
   \subseteq t_{\T_n}(M)=M
\]
is the unique brick chain filtration whose type is a subsequence of
$(B_1,\ldots,B_n)$. Every brick chain filtration is obtained from a
maximal torsion chain in this way.
\end{lemma}

\begin{proof}
Fix a cover $\T'\lessdot\T$ with label $B$. Apply
Proposition~\ref{prop:ringel-chains}\ref{item:ringel-label} to the
module $t_\T(M)\in\T$. It gives a submodule $X\in\T'$ with
$t_\T(M)/X\in\Filt B$. Since $\T'\subseteq{}^\perp B$, extension
closure of Hom-vanishing gives $\Hom_A(Y,Z)=0$ for
$Y\in\T'$ and $Z\in\Filt B$. Thus
Lemma~\ref{lem:torsion-part-characterization} gives
$X=t_{\T'}(t_\T(M))=t_{\T'}(M)$; the last equality uses
$\T'\subseteq\T$.
The nonzero factors of the evaluated chain therefore belong to the
categories $\Filt B_i$, in the order of their labels. These labels form
a complete brick chain by
Proposition~\ref{prop:ringel-chains}\ref{item:ringel-complete}, so
this is a brick chain filtration. Its uniqueness follows from
part~\ref{item:ringel-unique} of the same proposition.
Finally, extend the type of any given filtration to a complete chain
using part~\ref{item:ringel-complete}. Evaluating the corresponding
maximal torsion chain recovers that filtration by uniqueness.
\end{proof}

\begin{proposition}\label{prop:filtration-edges}
Let $A$ be $\tau$-tilting finite and let $M\in\modu A$.
\begin{enumerate}
\item\label{item:solid-vertices}
The map $U\mapsto[U]$ is a bijection from $\Sol_A(M)$ to the
vertices of $P(M)$.
\item\label{item:filtration-edges}
For submodules $U\subsetneq V$ of $M$, the terms $U,V$ occur
consecutively in some brick chain filtration if and only if $[U],[V]$
are the endpoints of an increasing edge of $P(M)$.
\end{enumerate}
In particular, brick chain filtrations inject into the increasing
edge paths of $P(M)$.
\end{proposition}

\begin{proof}
(1) Proposition~\ref{prop:torsion-parts} and
Theorem~\ref{thm:polytope-torsion}\ref{item:polytope-vertices} show
that the map has precisely the vertices as its image. To prove
injectivity, suppose $[U]=[V]=x$ is a vertex and $U\ne V$.
The exact sequence
\[
 0\longrightarrow U\cap V\longrightarrow U\oplus V
   \longrightarrow U+V\longrightarrow0
\]
gives $2x=[U\cap V]+[U+V]$. The two vectors on the right are distinct,
since $U\cap V$ is a proper submodule of $U$ and $U$ is a proper
submodule of $U+V$. Thus $x$ is the midpoint of two distinct points
of $P(M)$, contradicting extremality. This also shows that any submodule
whose composition vector is a vertex is uniquely determined by that vector.

(2) By Lemma~\ref{lem:evaluate-chain}, a given filtration is obtained
by evaluating a maximal torsion chain. Every consecutive pair
therefore comes from a cover on which the torsion part changes, and
Theorem~\ref{thm:polytope-torsion}\ref{item:polytope-edges} gives an edge.
Conversely, lift an increasing edge to a cover, extend that cover to a
maximal torsion chain, and apply Lemma~\ref{lem:evaluate-chain}.
The lifted cover gives
consecutive distinct terms, which must be $U,V$ by vertex uniqueness.
The same uniqueness shows that the edge path determines the actual
filtration, proving injectivity.
\end{proof}

For related descriptions of the vertices of $P(M)$, see
\cite[Theorem~4.2 and Proposition~4.7(1)]{FeiComb} and
\cite[Theorem~3.5 and Lemma~4.1]{AsaiIyama}.

The next example shows the polytope, its vertices and its increasing
edge paths for a module with two isomorphism classes of simple
composition factors, each occurring with multiplicity two.

\begin{example}\label{ex:pentagon}
Let $A$ be the path algebra of the quiver $1\to2$, and let $X$ be the
indecomposable projective module with top $S_1$; as a representation,
$X$ is $k\to k$ with the identity map, and its socle is $S_2$. Put
$M=X\oplus S_1\oplus S_2$, so that $[M]=(2,2)$ in the basis
$([S_1],[S_2])$. The torsion classes of $A$ are
\[
 0,\quad \add S_2,\quad \add S_1,\quad \add(X\oplus S_1),\quad\modu A.
\]
Their respective torsion parts in $M$ are
\[
 0,\quad S_2^2,\quad S_1,\quad X\oplus S_1,\quad M,
\]
with composition vectors $(0,0)$, $(0,2)$, $(1,0)$, $(2,1)$, $(2,2)$.
By Proposition~\ref{prop:filtration-edges}\ref{item:solid-vertices},
these five submodules are the solid submodules of $M$, and their
vectors are the vertices of $P(M)$:
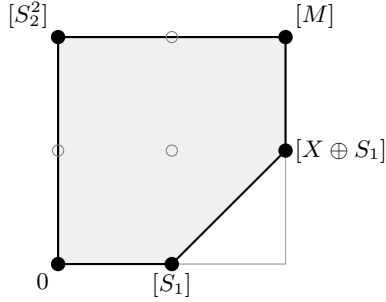
\[
\begin{tikzpicture}[scale=1.5,baseline=(current bounding box.center)]
  \draw[step=1,gray!35,thin] (0,0) grid (2,2);
  \draw[gray!70,thin] (0,0) rectangle (2,2);
  \fill[gray!12] (0,0)--(1,0)--(2,1)--(2,2)--(0,2)--cycle;
  \draw[thick] (0,0)--(0,2)--(2,2);
  \draw[thick] (0,0)--(1,0)--(2,1)--(2,2);
  \foreach \p/\l/\pos in {(0,0)/$0$/below left,(1,0)/$[S_1]$/below,(2,1)/$[X\oplus S_1]$/right,(2,2)/$[M]$/above right,(0,2)/$[S_2^2]$/above left}
    {\fill \p circle (1.8pt); \node[\pos,font=\small] at \p {\l};}
  \foreach \p in {(0,1),(1,1),(1,2)} {\draw[gray] \p circle (1.4pt);}
\end{tikzpicture}
\]
The polytope has five edges, all increasing, and exactly two increasing
edge paths from $0$ to $[M]$: the path through $[S_2^2]$ and the path
through $[S_1]$ and $[X\oplus S_1]$.
Evaluating the two maximal chains of torsion classes by
Lemma~\ref{lem:evaluate-chain} gives the brick chain filtrations
\[
 0\subsetneq S_2^2\subsetneq M
 \qquad\text{and}\qquad
 0\subsetneq S_1\subsetneq X\oplus S_1\subsetneq M,
\]
of types $(S_2,S_1)$ and $(S_1,X,S_2)$.
Since there are only two increasing edge paths, the injection in
Proposition~\ref{prop:filtration-edges} shows that $M$ has no other brick
chain filtration. The box $[0,(2,2)]$ contains $4!/(2!\,2!)=6$ lattice
paths with unit coordinate steps from $0$ to $(2,2)$, which is the value
of the bound in Theorem~\ref{thm:multinomial} for $M$.
\end{example}

The injection need not be surjective: an increasing path can violate
the Hom condition between nonadjacent factors. Fei observes that
consecutive edge factors satisfy forward Hom-vanishing, whereas
nonconsecutive factors need not
\cite[Example~8.8 and the preceding discussion]{Fei}.
The following example exhibits this obstruction for a
representation-finite algebra.

\begin{example}\label{ex:path-not-filtration}
Let $A$ be the path algebra of an oriented three-cycle modulo all
paths of length four. This is a representation-finite Nakayama algebra.
Let $M$ be an indecomposable projective, with simple factors
$S,T,R,S$ from socle to top, and let $U_i$ be its unique submodule of
length $i$. In the basis $([S],[T],[R])$, its submodule vectors
$v_i=[U_i]$ are
\[
 v_0=(0,0,0),\quad v_1=(1,0,0),\quad v_2=(1,1,0),\quad
 v_3=(1,1,1),\quad v_4=(2,1,1).
\]
Their convex hull is a pyramid with parallelogram base
$v_0,v_1,v_4,v_3$ in the plane $y=z$ and apex $v_2$.
The segments $[v_0,v_1]$ and $[v_3,v_4]$ are base edges.
Every segment joining the apex to a base vertex is an edge of the
pyramid, so $[v_1,v_2]$ and $[v_2,v_3]$ are edges as well.
Thus $v_0,v_1,v_2,v_3,v_4$ is an increasing edge path.
The corresponding composition series has successive factors $S,T,R,S$,
so it is not a brick chain filtration because $\Hom_A(S,S)\ne0$.
\end{example}

\subsection{Counting filtrations and the equality \texorpdfstring{$\bcount_A(M)=\ell(M)!$}{b(M)=length(M)!}}

We bound the number of increasing edge paths by encoding them as
lattice paths with coordinate unit steps. An edge need not be a
coordinate step: in Example~\ref{ex:pentagon}, the edge from $[S_1]$ to
$[X\oplus S_1]$ has direction $(1,1)$. We choose the encoding so
that the vertices of the original path are precisely the polytope
vertices visited by its encoding. This lets us recover the edge path
and gives an injection into a set of lattice paths that we can count.

\begin{lemma}\label{lem:lattice-paths}
Let $\mathbf d=(d_1,\ldots,d_r)\in\mathbb Z_{\geq0}^r$, and let $P$ be an
integral polytope contained in $[0,\mathbf d]$, with vertices $0,\mathbf d$. The number
of increasing edge paths from $0$ to $\mathbf d$ is at most
\[
 \frac{(d_1+\cdots+d_r)!}{d_1!\cdots d_r!}.
\]
\end{lemma}

\begin{proof}
For an increasing edge $[u,v]$, choose a linear functional $\theta$
whose maximum on $P$ is attained exactly on that edge. If $P$ is
itself a segment, we may take $\theta=0$.
Replace the edge by $v_i-u_i$ copies of the coordinate step $e_i$,
placing steps with positive $\theta$-value first, then those with value
zero, then those with negative value. Fix one such order for each edge.
Since $\theta(u)=\theta(v)$, the total of the step values is zero.
Before the negative steps, all partial sums are nonnegative.
During the negative steps, a partial sum is the negative of the sum
of the remaining nonpositive values, so it is again nonnegative.
Every intermediate lattice point $x$ therefore satisfies
$\theta(x)\geq\theta(u)=\max_P\theta$, whether or not $x$ lies in $P$.
All vertices of $P$ other than
$u,v$ have strictly smaller value, so the replacement meets no other
vertex of $P$.

Concatenate the replacements along an increasing edge path. The
polytope vertices visited by the resulting lattice path recover the
original path, proving injectivity. The number of lattice paths from
$0$ to $\mathbf d$ with steps $e_i$ is the displayed multinomial coefficient.
For $\mathbf d=0$, both sets consist of the empty path.
\end{proof}

The following theorem bounds the number of brick chain filtrations by
a multinomial coefficient in the simple composition multiplicities.

\begin{theorem}\label{thm:multinomial}
Let $A$ be a $\tau$-tilting finite algebra over a field $k$, and let
$S_1,\ldots,S_r$ represent its simple modules. Let $M$ have multiplicity
$d_i$ for $S_i$ and length $m=d_1+\cdots+d_r$. Then
\[
 \bcount_A(M)\leq\frac{m!}{d_1!\cdots d_r!}\leq m!.
\]
\end{theorem}

\begin{proof}
Use the injection in Proposition~\ref{prop:filtration-edges} and apply
Lemma~\ref{lem:lattice-paths} to $P(M)\subseteq[0,[M]]$.
\end{proof}

In particular, $\bcount_A(M)\leq\ell(M)!$ for every module $M$ over a
representation-finite algebra.

When exactly two isomorphism classes of simple modules occur as
composition factors, the polytope gives an exact count.

\begin{corollary}\label{cor:two-simple-count}
Let $A$ be $\tau$-tilting finite and let exactly two isomorphism
classes of simple modules occur as composition factors of $M$. Then $\bcount_A(M)=2$.
\end{corollary}

\begin{proof}
The submodule vectors in a composition series lie in $P(M)$, and
their successive differences include both corresponding simple basis
vectors. Since $0\in P(M)$, its affine hull therefore contains the
coordinate plane spanned by these two vectors. All submodule vectors
lie in that plane, so $P(M)$ is a polygon, with $0,[M]$ as its
coordinatewise minimum and maximum. Its edge graph is a cycle.
An increasing path cannot revisit a vertex, so it must follow one
of the two boundary arcs from $0$ to $[M]$.
There are thus at most two such paths, so
Proposition~\ref{prop:filtration-edges} gives $\bcount_A(M)\leq2$.
Ringel proves that a module whose simple composition factors are not
all isomorphic has at least two brick chain filtrations
\cite[Proposition~10.3]{RingelFiltrations}, giving equality.
\end{proof}

In particular, for such a module the injection of
Proposition~\ref{prop:filtration-edges} is bijective.
Example~\ref{ex:path-not-filtration} shows that the injection can fail to be
surjective when three isomorphism classes of simple modules occur as
composition factors.

Example~\ref{ex:semisimple} shows that
$m!$ is attained by a direct sum of $m$ pairwise nonisomorphic simple
modules. We next prove the bound $m!/2+(m-1)!$ for nonsemisimple modules
over $\tau$-tilting finite algebras over algebraically closed fields.

A module is \emph{multiplicity-free} if every simple composition
multiplicity is at most one. If some multiplicity exceeds one,
Theorem~\ref{thm:multinomial} already gives the stronger bound $m!/2$.
It therefore remains to treat nonsemisimple multiplicity-free modules.
By Morita equivalence, we may work with a basic algebra. Over an
algebraically closed field, it has a presentation $kQ/I$ with $I$
admissible, and a multiplicity-free module has spaces $Me_i$ of
dimension at most one. Its submodules are determined by their supports.

Attach to such a module $M$ the directed graph $G$ whose vertices are
the supported vertices of $Q$ and whose arrows are those acting
nontrivially on $M$. This graph is acyclic. Indeed, each arrow acts
by a nonzero map between one-dimensional spaces. A directed cycle
would therefore give nonzero composites of arbitrarily large length,
whereas admissibility of $I$ makes every sufficiently long path act
as zero. We count filtrations by partitions of this graph.

For a finite acyclic directed graph $G$, allowing parallel arrows,
let $f(G)$ be the number of
ordered partitions of its vertices into nonempty blocks such that each
block induces a connected underlying undirected graph, and every arrow
between distinct blocks points from a later block to an earlier block.
The next lemma compares these counts when one arrow is deleted.

\begin{lemma}\label{lem:arrow-deletion}
If $e:x\to y$ is an arrow of a finite acyclic directed graph $G$, then
$f(G)\leq f(G-e)$.
\end{lemma}

\begin{proof}
Keep an ordered partition unchanged if all its blocks remain connected
after deleting $e$. Otherwise, the block containing $x,y$ splits into
exactly two connected components $C_x,C_y$. Replace that block by
$C_x,C_y$, consecutively in this order. No remaining arrow joins these
components, and the order relative to every other block is unchanged.
Thus the new partition is counted by $f(G-e)$.

This map is injective. In an unchanged image, the block containing $x$
is at or after the block containing $y$. In a split image it is
strictly before that block. Hence the two kinds of image are disjoint,
and a split image is recovered by merging the consecutive blocks
containing $x,y$.
\end{proof}

\begin{proposition}\label{prop:nonsemisimple}
Let $A$ be $\tau$-tilting finite over an algebraically closed field,
and let $M$ have length $m$.
\begin{enumerate}
\item\label{item:nonsemisimple-bound}
If $M$ is nonsemisimple, then $m\geq2$ and
\[
 \bcount_A(M)\leq\frac{m!}{2}+(m-1)!.
\]
For nonsemisimple multiplicity-free $M$, this inequality holds without
assuming that $A$ is $\tau$-tilting finite.
For every $m\geq2$, the bound is attained over a representation-finite
algebra.
\item\label{item:factorial-equality}
If $m\geq3$, then $\bcount_A(M)=m!$ if and only if $M$ is a direct
sum of pairwise nonisomorphic simple modules.
\end{enumerate}
\end{proposition}

\begin{proof}
(1) If some simple composition multiplicity is at least two,
Theorem~\ref{thm:multinomial} gives $\bcount_A(M)\leq m!/2$.
We may therefore assume that $M$ is multiplicity-free.
Use the bound quiver presentation and the acyclic graph $G$ attached
to $M$ above; Morita equivalence preserves the filtration count.

A submodule support is closed under the arrows of $G$. Thus a
filtration gives an ordered partition whose initial unions are closed
under arrows. This is equivalent to requiring every arrow between
blocks to point from a later block to an earlier block. Conversely,
such a partition determines a unique chain of submodules, by taking
these initial unions as supports.

A factor of a filtration is still multiplicity-free, so it belongs to $\Filt B$ for
a brick $B$ only if it is itself isomorphic to $B$. The nonzero arrow
maps of a factor are precisely the arrows of $G$ with both endpoints
in its block. Thus its graph is the subgraph of $G$ induced on that
block. Its endomorphisms are scalars constant on each connected
component, so the factor is a brick exactly when this subgraph is
connected. Different factors have disjoint simple supports and hence
no nonzero morphisms between them. It follows that brick chain
filtrations correspond exactly to the ordered partitions counted by
$f(G)$.

Since $M$ is nonsemisimple, $G$ contains an arrow. Repeatedly apply
Lemma~\ref{lem:arrow-deletion} until just one arrow $x\to y$ remains.
Writing $G'$ for the resulting graph, a partition counted by $f(G')$
either has all blocks singletons, with $y$ before $x$, or has the block
$\{x,y\}$ and all
other blocks singletons. These possibilities contribute $m!/2$ and
$(m-1)!$, respectively, proving the bound.

For sharpness, take the linearly oriented quiver of Dynkin type $A_m$ and the
representation with dimension one at every vertex, one arrow equal to
the identity and all other arrows zero. Its graph $G$ has just one
arrow, so its filtration count is $m!/2+(m-1)!$.

(2) If $M$ is nonsemisimple, (1) gives $\bcount_A(M)<m!$ for $m\geq3$.
If $M$ is semisimple and has a repeated simple factor,
Theorem~\ref{thm:multinomial} again gives $\bcount_A(M)<m!$.
Conversely, a direct sum of $m$ distinct
simples has $m!$ filtrations by Example~\ref{ex:semisimple}.
\end{proof}

\section{Counting brick chain filtrations of Kronecker modules}\label{sec:kronecker}

We construct bricks over generalized Kronecker algebras whose numbers
of brick chain filtrations exceed the factorial of their composition
length, answering Question~\ref{ques:ringel}(2) affirmatively.
We obtain a module of composition length $244$, as well as families
over the three-arrow algebra
whose filtration counts grow exponentially in the square of their
composition length. To do this, we count submodules $U\subseteq M$
for which $U$ and $M/U$ are bricks and $\Hom_A(U,M/U)=0$.
Each gives a two-step brick chain filtration. Schubert calculus computes
the number of these submodules; the hook-length formula then gives
lower bounds for the number of filtrations.

Throughout this section, $k$ is algebraically closed. Let $Q_q$ have
vertices $1,2$ and arrows $1\to2$, numbered $1,\ldots,q$, where $q\geq3$.
Put $A=k Q_q$. We identify a representation with the right $A$-module
whose vertex spaces are $Me_i$ and whose arrow maps are right
multiplication by the arrows. For representations $X,Y$ of a quiver,
a morphism $\phi:X\to Y$ consists of maps $\phi_i:X_i\to Y_i$ with
$Y(c)\phi_i=\phi_jX(c)$ for every arrow $c:i\to j$.
The two simple $k Q_q$-modules are one-dimensional, so a module of
dimension vector $(x,y)$ has composition length $x+y$.
A property holds for a \emph{general representation} of a fixed
dimension vector if it holds on a nonempty Zariski-open subset of the
corresponding representation variety.

We next choose the dimension vectors used in the count. For an acyclic
quiver $Q$, its Euler form is
\begin{equation}\label{eq:euler-form}
 \langle x,y\rangle_Q
   =\sum_{i\in Q_0}x_i y_i-\sum_{c:i\to j}x_i y_j.
\end{equation}
For representations $X,Y$ of dimension vectors $x,y$, respectively,
the hereditary algebra $kQ$ satisfies \cite[Section~2]{DSW}
\begin{equation}\label{eq:euler-hom-ext}
 \langle x,y\rangle_Q
   =\dim_k\Hom_{kQ}(X,Y)-\dim_k\Ext^1_{kQ}(X,Y).
\end{equation}
For $Q_q$, the Euler form is
\[
 \langle x,y\rangle_{Q_q}=x_1y_1+x_2y_2-qx_1y_2.
\]
For positive integers $a,b$, set
\begin{equation}\label{eq:dimensions}
 \beta=(a,a),\qquad \gamma=((q-1)b,b),\qquad
 \alpha=\beta+\gamma=(a+(q-1)b,a+b).
\end{equation}
The choice in \eqref{eq:dimensions} satisfies
\[
 \langle\beta,\gamma\rangle_{Q_q}
   =a(q-1)b+ab-qab=0.
\]
Here a subrepresentation of dimension $\beta$ has quotient dimension
$\gamma$. Its vertex subspaces have $\sum_i\beta_i\gamma_i$ parameters,
and the arrows impose $\sum_{c:i\to j}\beta_i\gamma_j$ equations.
Their difference is $\langle\beta,\gamma\rangle$. Its vanishing is
the condition in
Proposition~\ref{prop:genericity} that makes the submodule count finite
for a general representation of dimension $\alpha$. The proposition
also gives $\Hom_A(U,M/U)=0$ for all the submodules being counted.

\subsection{Submodules of general representations with brick factors}

We first show that, for a general representation of dimension $\alpha$,
every submodule of dimension $\beta$ gives a two-step brick chain
filtration with brick factors. We use a result on subrepresentations of
general representations of acyclic quivers, then verify its hypotheses
for the dimension vectors in \eqref{eq:dimensions}.

For a finite acyclic quiver $Q$ and a representation $M$, write
$\Gr_\beta(M)$ for the \emph{quiver Grassmannian}. To keep track of
intersection multiplicities, we now equip it with the closed subscheme
structure in $\prod_i\Gr(\beta_i,M_i)$
defined by the subrepresentation equations $M(c)U_i\subseteq U_j$ for
each arrow $c:i\to j$. On affine charts these are polynomial equations,
obtained by setting the induced maps $U_i\to M_j/U_j$ equal to zero.
Its $k$-points are the subrepresentations of $M$ of dimension vector
$\beta$. For $A=kQ$, this set is the part of
$\Gr_A(\sum_i\beta_i,M)$ with vertex dimensions $\beta_i$.
For a finite scheme $Z$ over the
algebraically closed field $k$, its degree is
$\deg Z=\sum_{z\in Z}\dim_k\mathcal O_{Z,z}$.
If $Z$ is reduced, every local ring is $k$, so its degree equals its
number of points.

The following proposition states the result of Derksen, Schofield and
Weyman on subrepresentations of a general representation in the form
used here. In particular, it defines the integer $N(\beta,\alpha)$.

\begin{proposition}[{\cite[p.~139, proof of Corollary~12 and
Proposition~13]{DSW}; \cite[proof of Corollary~3]{CB}}]\label{prop:genericity}
Let $Q$ be a finite acyclic quiver, and let
$\alpha=\beta+\gamma$ with $\langle\beta,\gamma\rangle_Q=0$.
There is a nonempty Zariski-open subset $\mathcal O\subseteq\Rep(Q,\alpha)$
with the following properties.
\begin{enumerate}
\item\label{item:generic-count}
For every $M\in\mathcal O$, the quiver Grassmannian $\Gr_\beta(M)$
is finite and reduced, with constant cardinality $N(\beta,\alpha)$.
\item\label{item:generic-hom}
For every $M\in\mathcal O$ and every $U\in\Gr_\beta(M)$,
one has $\Hom_{kQ}(U,M/U)=0$.
\item\label{item:generic-open}
Given nonempty open subsets
$\mathcal U_\beta\subseteq\Rep(Q,\beta)$ and
$\mathcal U_\gamma\subseteq\Rep(Q,\gamma)$ invariant under change of
bases, $\mathcal O$ can be chosen so that, for every $M\in\mathcal O$
and every $U\in\Gr_\beta(M)$, the submodule $U$ and its quotient
$M/U$ belong to $\mathcal U_\beta$ and $\mathcal U_\gamma$,
respectively.
\end{enumerate}
\end{proposition}

The finiteness and constant cardinality in part~\ref{item:generic-count}
are established in the paragraph following equation~(5) in
\cite[p.~139]{DSW}. The proof of \cite[Corollary~3]{CB} shows that the
quiver Grassmannian is generically reduced for a general representation.
Since it is zero-dimensional here, it is reduced.
Part~\ref{item:generic-hom} is established in
the proof of \cite[Corollary~12]{DSW}, and
part~\ref{item:generic-open} is \cite[Proposition~13]{DSW}.

The value $N(\beta,\alpha)=0$ is allowed in
Proposition~\ref{prop:genericity}.

We will impose the condition that the submodules and quotients are
bricks. Since $k$ is algebraically closed, a representation is a brick
precisely when its endomorphism space has dimension one. In each
representation variety
this locus is open: endomorphisms are the kernel of a matrix depending
polynomially on the arrow maps, so the condition $\dim\End\leq1$ is open.
For a nonzero representation the identity endomorphism gives
$\dim\End\geq1$, and hence this condition is equivalent to $\dim\End=1$.
It is invariant under changes of bases because it depends only on
the isomorphism class. It remains to show that the brick loci for
$\beta,\gamma$ and $\alpha$ are nonempty.

\begin{lemma}\label{lem:brick-pairs}
For the vectors in \eqref{eq:dimensions}, there are bricks $U,V$ with
dimension vectors $\beta,\gamma$ satisfying the following properties.
\begin{enumerate}
\item\label{item:brick-pair-hom} One has
\[
 \Hom_{k Q_q}(U,V)=\Hom_{k Q_q}(V,U)=0.
\]
\item\label{item:brick-pair-ext} One has $\Ext^1_{k Q_q}(V,U)\ne0$, and every nonsplit exact sequence
\[
 0\longrightarrow U\longrightarrow M\longrightarrow V\longrightarrow0
\]
has a brick $M$ of dimension vector $\alpha$ as its middle term.
\end{enumerate}
In particular, the brick locus in $\Rep(Q_q,\alpha)$ is nonempty.
\end{lemma}

\begin{proof}
Let $D_a$ be a diagonal matrix with pairwise distinct entries, and let
$E_a$ be the $a\times a$ matrix all of whose entries are one. Define $U$
by the arrow maps
\[
 I_a,\ D_a,\ E_a,\ 0,\ldots,0.
\]
An endomorphism has equal source and target components because of the
first arrow. Commuting with $D_a$ forces this common component to be
diagonal, and commuting with $E_a$ forces it to be scalar. Thus $U$ is
a brick.

Define $V$ on the spaces $(k^b)^{q-1}$ and $k^b$. Its first $q-1$
arrows are the coordinate projections $P_1,\ldots,P_{q-1}$, and its
last arrow is
\[
 [D_b+tI_b\quad E_b\quad 0\quad\cdots\quad0],
\]
where $D_b,E_b$ are chosen in the same way and $t\in k$ will be fixed
below. If $(X,Y)$ is an endomorphism of $V$, our morphism convention
gives $P_jX=YP_j$ for the projection arrows, hence
$X=\operatorname{diag}(Y,\ldots,Y)$. The last arrow then gives
$YD_b=D_bY$ and $YE_b=E_bY$, so $Y$ is scalar. Hence $V$ is a brick
for every $t$.

(1) For a morphism $V\to U$, let $X,Y$ be its source and target components.
The first arrow gives $X=[Y\quad0\quad\cdots\quad0]$.
The second arrow gives $D_aX=YP_2$. Its second block gives $Y=0$, and then
$X=0$. Thus $\Hom(V,U)=0$.
For a morphism $U\to V$, write
$X:k^a\to(k^b)^{q-1}$ and $Y:k^a\to k^b$ for its source and target
components, and let $G_j$ be the $j$-th arrow matrix of $U$. The first
$q-1$ arrow equations are
$P_jX=YG_j$. Hence the block rows of $X$ are
$YG_1,\ldots,YG_{q-1}$.
The last equation is
\[
 (D_b+tI_b)Y+E_bYD_a=YG_q.
\]
Thus it is $tY+L(Y)=0$, where
\[
 L(Y)=D_bY+E_bYD_a-YG_q,\qquad
 G_q=\begin{cases}E_a&q=3,\\0&q\geq4.\end{cases}
\]
This is an endomorphism of the $ab$-dimensional space
$\Hom_{k}(k^a,k^b)$, independent of $t$. The polynomial
$\det(tI+L)$ is monic. Since $k$ is infinite, choose $t$ outside
its finite set of roots.
Then $Y=0$, and the block equations give $X=0$, so
$\Hom(U,V)=0$ as well.

(2) Apply \eqref{eq:euler-hom-ext} to $V,U$.
Since $\Hom_{kQ_q}(V,U)=0$ by (1), it gives
\[
 \dim\Ext^1_{k Q_q}(V,U)
   =-\langle\gamma,\beta\rangle_{Q_q}=q(q-2)ab>0.
\]
The last inequality uses $q\geq3$ and $a,b\geq1$.
Choose a nonsplit extension $0\to U\xrightarrow{i}M\xrightarrow{p}V\to0$.
Every endomorphism $f$ of $M$ preserves $U$, because $pfi=0$ by
$\Hom(U,V)=0$. It therefore induces a commutative diagram with exact rows
\[
\begin{tikzcd}
 0 \arrow[r] & U \arrow[r,"i"] \arrow[d,"u"']
 & M \arrow[r,"p"] \arrow[d,"f"]
 & V \arrow[r] \arrow[d,"v"] & 0 \\
 0 \arrow[r] & U \arrow[r,"i"']
 & M \arrow[r,"p"'] & V \arrow[r] & 0.
\end{tikzcd}
\]
Let
\[
 \rho:\End(M)\longrightarrow\End(U)\times\End(V),
 \qquad f\longmapsto(u,v).
\]
The map $\rho$ is injective: an element of its kernel factors through a
map $V\to U$. Indeed, $fi=0$ implies that $f$ factors through
$p:M\twoheadrightarrow V$, say $f=hp$ for a map $h:V\to M$. The
equality $pf=0$ gives $ph=0$ because $p$ is surjective, so $h$ factors
through $i:U\hookrightarrow M$. The resulting map $V\to U$ is zero because
$\Hom_{k Q_q}(V,U)=0$.
Write $u=\lambda\,1_U$ and $v=\mu\,1_V$. If
$\varepsilon\in\Ext^1(V,U)$ is the extension class, the commutative
diagram says that its pushout along $u$ and its pullback along $v$
coincide. Hence
\[
 \lambda\varepsilon=u_*\varepsilon=v^*\varepsilon
   =\mu\varepsilon.
\]
The extension is nonsplit, so $\varepsilon\ne0$ and $\lambda=\mu$.
The pair $(u,v)$ is therefore the image under $\rho$ of the scalar
endomorphism $\lambda\,1_M$. Injectivity of $\rho$ gives
$f=\lambda\,1_M$, and hence $\End(M)=k$.
\end{proof}

To distinguish the types of the filtrations we will count, we use a
general fact about solid submodules: a solid submodule has no distinct
isomorphic copy inside the same ambient module.

\begin{lemma}\label{lem:unique-submodule}
Let $A$ be a finite-dimensional algebra and $M$ an $A$-module.
If $U\in\Sol_A(M)$ and $U'\subseteq M$ is isomorphic to $U$, then
$U'=U$.
\end{lemma}

\begin{proof}
Let $j:U'\hookrightarrow M$ be the inclusion and
$p:M\twoheadrightarrow M/U$ the quotient map. Since $U'\simeq U$, we
have $\Hom_A(U',M/U)=0$. Thus $pj=0$, and $j$ factors through $U$:
\[
\begin{tikzcd}
 & U' \arrow[d,hook,"j"] \arrow[dl,dashed,"\overline j"'] & \\
 U \arrow[r,hook] & M \arrow[r,two heads,"p"] & M/U.
\end{tikzcd}
\]
Hence $U'\subseteq U$, and equality follows from their equal dimensions.
\end{proof}

\subsection{A Littlewood--Richardson formula for filtration counts}

We now count two-step brick chain filtrations whose first term has
dimension vector $\beta$, for a general representation of dimension
$\alpha$. We express the generic submodule count $N(\beta,\alpha)$
as a sum of squares of Littlewood--Richardson coefficients.

Let $\parts_{a,b}$ be the set of partitions
contained in the rectangle $b^a$, including the empty partition.
Thus its elements have at most $a$ rows and at most $b$ columns.
For a partition $\lambda$, let $s_\lambda$ be its Schur function, and
define the nonnegative integers $c(\lambda_1,\ldots,\lambda_r;\nu)$ by
\[
s_{\lambda_1}\cdots s_{\lambda_r}
   =\sum_\nu c(\lambda_1,\ldots,\lambda_r;\nu)s_\nu.
\]
Here $\lambda_1,\ldots,\lambda_r$ are the input partitions and $\nu$ is
the output partition. These are the Littlewood--Richardson coefficients
with any number of inputs. We first use them as tensor product
multiplicities for general linear groups; the hook-length estimate below uses
the same coefficients as restriction multiplicities for symmetric
groups. Put
\begin{equation}\label{eq:K}
 K(q,a,b)=
 \sum_{\substack{\nu,\lambda_1,\ldots,\lambda_{q-1}\in\parts_{a,b}}}
 c(\lambda_1,\ldots,\lambda_{q-1};\nu)^2.
\end{equation}

\begin{theorem}\label{thm:count}
For $q\geq3$ and $a,b\geq1$, there is a nonempty Zariski-open subset of
$\Rep(Q_q,\alpha)$ consisting of bricks $M$ with the following properties.
\begin{enumerate}
\item\label{item:filtration-count} The module $M$ has exactly $K(q,a,b)$
two-step brick chain filtrations whose first term has dimension vector
$\beta$.
\item\label{item:brick-factors} In every such filtration, the first term and the quotient are bricks.
\item\label{item:distinct-types} The first terms are pairwise nonisomorphic. In particular, these
filtrations have distinct types.
\end{enumerate}
\end{theorem}

Thus $\bcount_{k Q_q}(M)\geq K(q,a,b)$ on this open set.

\begin{proof}
Lemma~\ref{lem:brick-pairs} shows that the brick loci for
$\beta,\gamma$ and $\alpha$ are nonempty; their openness and invariance
were established above. Apply Proposition~\ref{prop:genericity}\ref{item:generic-open} to
the brick loci for $\beta$ and $\gamma$, and intersect the resulting
open set with the brick locus for $\alpha$. The intersection is nonempty
because $\Rep(Q_q,\alpha)$ is irreducible. For $M$ in it, every point
$U$ of $\Gr_\beta(M)$ gives bricks $U,M/U$ and satisfies
$\Hom_{k Q_q}(U,M/U)=0$ by Proposition~\ref{prop:genericity}\ref{item:generic-hom}.
Thus each point gives a two-step brick chain filtration and
$U\in\Sol_{k Q_q}(M)$.

(1) It remains to compute the number of points. We do this in two steps.

\textbf{Step 1: compute the submodule count by Schubert calculus.}
For $\rho\in\parts_{a,b}$, define its rectangular complement by
\[
 \overline\rho=(b-\rho_a,\ldots,b-\rho_1),
\]
where $\rho=(\rho_1,\ldots,\rho_a)$ is padded with zero parts.
The count takes place on the product
\[
 G_1=\Gr(a,a+(q-1)b),\qquad G_2=\Gr(a,a+b),
\]
which parametrizes the two vertex subspaces of a subrepresentation.
Recall that on $\Gr(a,a+c)$ the Schubert classes $\sigma_\lambda$
are indexed by partitions $\lambda\subseteq c^a$, with codimension
$|\lambda|$. The class $\sigma_{c^a}$ is the class of a point, and
multiplication is given by the Littlewood--Richardson coefficients;
terms indexed outside $c^a$ are zero. See \cite[Sections~1 and~3]{DSW}.
In particular, the coefficient
of the point class in a product is $c(\lambda_1,\ldots,\lambda_r;c^a)$.

Write $\sigma_\lambda^{(i)}$ for the class on $G_i$.
For an arrow map
$f:k^{a+(q-1)b}\to k^{a+b}$, let
\[
 D_f=\{(U_1,U_2)\in G_1\times G_2\mid f(U_1)\subseteq U_2\}.
\]
We give $D_f$ the closed subscheme structure defined by this incidence condition.
If the arrow maps of $M$ are $f_1,\ldots,f_q$, then
\[
 \Gr_\beta(M)=D_{f_1}\cap\cdots\cap D_{f_q}
 \quad\text{inside }G_1\times G_2.
\]
Crawley-Boevey proves the following two assertions
\cite[Theorem and proof, p.~365]{CB}; see also
\cite[Proposition~3]{DSW}. For a general arrow map $f$, the incidence
locus has class
\[
 [D_f]=\sum_{\lambda\in\parts_{a,b}}
       \sigma_\lambda^{(1)}\otimes\sigma_{\overline\lambda}^{(2)}
\]
in the Chow ring of $G_1\times G_2$. For a general tuple of arrow
maps $f_1,\ldots,f_q$, the cycle of their scheme-theoretic intersection
is the product of these classes. After further shrinking the open subset
of $\Rep(Q_q,\alpha)$ chosen at the start of the proof, we may assume
that both assertions hold for the arrow maps of $M$.
Since each arrow contributes the
same class, we obtain
\[
 [\Gr_\beta(M)]=\prod_{j=1}^q[D_{f_j}]
 =\left(\sum_{\lambda\in\parts_{a,b}}
       \sigma_\lambda^{(1)}\otimes\sigma_{\overline\lambda}^{(2)}\right)^q.
\]
By Proposition~\ref{prop:genericity}\ref{item:generic-count}, this
scheme is finite and reduced for general $M$, so its cycle has
coefficient one at each point. Its cardinality $N(\beta,\alpha)$ is therefore
its degree, the coefficient of the point class in the displayed expression.
Expanding the power and multiplying Schubert classes in each factor
therefore gives
\begin{equation}\label{eq:dsw-count}
 N(\beta,\alpha)=
 \sum_{\lambda_1,\ldots,\lambda_q\in\parts_{a,b}}
 c(\lambda_1,\ldots,\lambda_q;((q-1)b)^a)
 c(\overline\lambda_1,\ldots,\overline\lambda_q;b^a).
\end{equation}
The partitions $((q-1)b)^a$ and $b^a$ index the point classes in
$G_1$ and $G_2$. This explains both rectangles in the formula.

\textbf{Step 2: rewrite the count as a sum of squares.}
The expression in \eqref{eq:dsw-count} is an integer independent of the
algebraically closed field $k$. To simplify it, we use complex
representations of general linear groups to prove identities between
its Littlewood--Richardson coefficients.
Let $V_\lambda$ be the irreducible polynomial representation of
$\GL_a(\C)$ of highest weight $\lambda$, for a partition $\lambda$
with at most $a$ parts. Its character is $s_\lambda(x_1,\ldots,x_a)$;
thus $c(\lambda_1,\ldots,\lambda_r;\nu)$ is the multiplicity of
$V_\nu$ in $V_{\lambda_1}\otimes\cdots\otimes V_{\lambda_r}$
for partitions with at most $a$ parts; see \cite[Section~1]{PPY}.
For a nonnegative integer $c$, write $\det^c$ for the $c$th power of
the determinant representation. The rectangular partition $c^a$ satisfies
$V_{c^a}\simeq\det^c$, since both representations have highest weight
$(c,\ldots,c)$.
The dual $V_\lambda^*$ has highest weight
$(-\lambda_a,\ldots,-\lambda_1)$. Tensoring with $\det^b$ adds $b$
to every coordinate, giving the highest weight
$\overline\lambda$. Thus
\begin{equation}\label{eq:rectangle-duality}
 V_{\overline\lambda}\simeq\det^b\otimes V_\lambda^*.
\end{equation}
Put $X=V_{\lambda_1}\otimes\cdots\otimes V_{\lambda_q}$. Applying
\eqref{eq:rectangle-duality} to all $q$ factors gives
\[
 V_{\overline\lambda_1}\otimes\cdots\otimes V_{\overline\lambda_q}
   \simeq\det^{qb}\otimes X^*.
\]
Thus the second coefficient in \eqref{eq:dsw-count} is the multiplicity
of $\det^b$ in $\det^{qb}\otimes X^*$.
Tensor-Hom adjunction, followed by tensoring both arguments with
$\det^{-b}$, gives
\begin{align*}
 \Hom_{\GL_a}(\det^b,\det^{qb}\otimes X^*)
 &\simeq\Hom_{\GL_a}(\det^b\otimes X,\det^{qb})\\
 &\simeq\Hom_{\GL_a}(X,\det^{(q-1)b}).
\end{align*}
Since finite-dimensional representations of $\GL_a(\C)$ are completely
reducible, this dimension is the multiplicity of
$\det^{(q-1)b}$ in $X$, namely the first coefficient in
\eqref{eq:dsw-count}. Each product in that sum is therefore the square
of its second coefficient. Reindex by $\mu_i=\overline\lambda_i$ for
every $i$. Rectangular complementation is an involution of
$\parts_{a,b}$, so the new indices run over the same set, giving
\[
 N(\beta,\alpha)=
 \sum_{\mu_1,\ldots,\mu_q\in\parts_{a,b}}
 c(\mu_1,\ldots,\mu_q;b^a)^2.
\]
Finally, put $W=V_{\mu_1}\otimes\cdots\otimes V_{\mu_{q-1}}$.
Tensor-Hom adjunction and \eqref{eq:rectangle-duality} give
\begin{align*}
 c(\mu_1,\ldots,\mu_q;b^a)
 &=\dim\Hom_{\GL_a}(\det^b,W\otimes V_{\mu_q})\\
 &=\dim\Hom_{\GL_a}(\det^b\otimes V_{\mu_q}^*,W)\\
 &=c(\mu_1,\ldots,\mu_{q-1};\overline\mu_q).
\end{align*}
Reindex by $\lambda_i=\mu_i$ for $1\leq i<q$ and
$\nu=\overline\mu_q$. Since complementation permutes $\parts_{a,b}$,
the resulting sum is precisely \eqref{eq:K}, so
$N(\beta,\alpha)=K(q,a,b)$. The summand with
$\nu=\lambda_1=\cdots=\lambda_{q-1}=\varnothing$ equals one. We have
therefore proved $N(\beta,\alpha)=K(q,a,b)>0$.

(2) The open set was chosen so that every submodule of dimension
vector $\beta$ and its quotient are bricks.

(3) Every such submodule is solid in $M$, so
Lemma~\ref{lem:unique-submodule} shows that distinct first terms are
nonisomorphic. The first brick in the type is the first term itself
by (2), so the filtrations have distinct types.
\end{proof}

\begin{example}\label{ex:small-count}
For $a=b=1$, the elements of $\parts_{1,1}$ are $\varnothing$ and $(1)$.
The term with $\nu=\varnothing$ contributes one to \eqref{eq:K}. For
$\nu=(1)$, the contribution is $q-1$, since exactly one of
$\lambda_1,\ldots,\lambda_{q-1}$ equals $(1)$. Thus $K(q,1,1)=q$.
A general representation of dimension $(q,2)$ consequently has exactly
$q$ two-step brick chain filtrations whose first term has dimension vector
$(1,1)$. Since $q\geq3$, this shows that the $\tau$-tilting
finiteness assumption in Corollary~\ref{cor:two-simple-count} cannot
be omitted.
\end{example}

\subsection{A hook-length lower bound}

We now bound the number of brick chain filtrations supplied by
Theorem~\ref{thm:count}\ref{item:filtration-count}. We estimate its Littlewood--Richardson sum
$K(q,a,b)$ using dimensions of representations of symmetric groups and
the hook-length formula, without evaluating the sum explicitly.

For a partition $\nu$, let $|\nu|$ be its size and let $\nu'$ be its
transpose. The hook length of a cell $(i,j)\in\nu$ is
\[
 h_\nu(i,j)=\nu_i-j+\nu'_j-i+1,
 \qquad H(\nu)=\prod_{(i,j)\in\nu}h_\nu(i,j).
\]
Let $S^\nu$ be the complex Specht module indexed by $\nu$, and put
$f^\nu=\dim S^\nu$. This is also the number of standard Young tableaux
of shape $\nu$. The hook-length formula and the decomposition of the
regular representation of the symmetric group give
\begin{equation}\label{eq:specht-identities}
 f^\nu=\frac{|\nu|!}{H(\nu)},
 \qquad \sum_{\lambda\vdash d}(f^\lambda)^2=d!,
\end{equation}
respectively; see \cite[Sections~2.2--2.3]{PPY}.

The Littlewood--Richardson coefficients have a second role which
relates the count in \eqref{eq:K} to these dimensions. If
$|\nu|=d=e_1+\cdots+e_r$ with $e_i\geq0$, restriction of $S^\nu$ to the
subgroup $S_{e_1}\times\cdots\times S_{e_r}$ contains
$S^{\lambda_1}\otimes\cdots\otimes S^{\lambda_r}$ with multiplicity
$c(\lambda_1,\ldots,\lambda_r;\nu)$ for $\lambda_i\vdash e_i$.
Here each symmetric group acts
on its corresponding tensor factor, and $S_0$ is the trivial group.
This is the Littlewood--Richardson rule for representations of
symmetric groups; the case $r=2$ in \cite[Section~4.1]{PPY} gives the
general case by restriction along the successive Young subgroups
\[
 S_d\supseteq S_{e_1}\times S_{d-e_1}
 \supseteq S_{e_1}\times S_{e_2}\times S_{d-e_1-e_2}
 \supseteq\cdots.
\]
Taking dimensions yields
\begin{equation}\label{eq:restriction-dimensions}
 f^\nu=\sum_{\lambda_i\vdash e_i}
 c(\lambda_1,\ldots,\lambda_r;\nu)\prod_{i=1}^r f^{\lambda_i}.
\end{equation}
We use this equality to bound a sum of squares of the multiplicities.

\begin{proposition}\label{prop:hook-bound}
For any $\nu\in\parts_{a,b}$ of size $d$, one has
\begin{equation}\label{eq:hook-bound}
 K(q,a,b)\geq\frac{(q-1)^d d!}{H(\nu)^2}.
\end{equation}
\end{proposition}

\begin{proof}
Put $r=q-1$ and fix nonnegative integers $e_i$ with
$e_1+\cdots+e_r=d$. Apply Cauchy--Schwarz to
\eqref{eq:restriction-dimensions}, then use
\eqref{eq:specht-identities}:
\begin{align*}
 (f^\nu)^2
 &\leq
 \left(\sum_{\lambda_i\vdash e_i}
     c(\lambda_1,\ldots,\lambda_r;\nu)^2\right)
 \left(\sum_{\lambda_i\vdash e_i}
     \prod_{i=1}^r(f^{\lambda_i})^2\right)\\
 &=\left(\sum_{\lambda_i\vdash e_i}
     c(\lambda_1,\ldots,\lambda_r;\nu)^2\right)
     \prod_{i=1}^r e_i!.
\end{align*}
Dividing by $\prod_i e_i!$ gives
\[
 \sum_{\lambda_i\vdash e_i}c(\lambda_1,\ldots,\lambda_r;\nu)^2
 \geq \frac{(f^\nu)^2}{e_1!\cdots e_r!}.
\]
For two inputs, the tableau form of the Littlewood--Richardson rule
\cite[Section~5.2, Proposition~3]{Fulton} counts tableaux on the skew diagram
$\nu/\lambda_1$. Thus
$c(\lambda_1,\lambda_2;\nu)\ne0$ requires
$\lambda_1\subseteq\nu$, and symmetry gives
$\lambda_2\subseteq\nu$. For several inputs, commutativity of Schur
functions lets us put any chosen $\lambda_i$ first. After reordering
so this is $\lambda_1$, expand successive products:
\[
 c(\lambda_1,\ldots,\lambda_r;\nu)
 =\sum_{\mu_2,\ldots,\mu_{r-1}}
   \prod_{j=2}^r c(\mu_{j-1},\lambda_j;\mu_j),
 \qquad \mu_1=\lambda_1,\quad\mu_r=\nu.
\]
All coefficients are nonnegative. If the left side is nonzero, some
nonzero summand therefore gives
\[
 \mu_1\subseteq\mu_2\subseteq\cdots\subseteq\mu_r,
\]
where each containment follows from the Littlewood--Richardson rule for
two factors. Hence every
$\lambda_i$ in a nonzero term above is contained in $\nu$, so belongs
to $\parts_{a,b}$.
The sets of tuples $(\lambda_1,\ldots,\lambda_r)$ for different
$(e_1,\ldots,e_r)$ are
disjoint. All summands in \eqref{eq:K} are nonnegative, so retaining the
fixed output $\nu$ and these disjoint tuples gives a lower bound for
$K(q,a,b)$. Summing the inequalities and using the multinomial theorem gives
\[
 K(q,a,b)\geq(f^\nu)^2
    \sum_{e_1+\cdots+e_r=d}\frac1{e_1!\cdots e_r!}
  =\frac{r^d(f^\nu)^2}{d!}.
\]
The hook-length formula in \eqref{eq:specht-identities} proves
\eqref{eq:hook-bound}.
\end{proof}

\subsection{Staircase partitions and quadratic growth}

We specialize to $q=3$ and $a=b=n$, so the modules under consideration
have dimension vector $(3n,2n)$ and composition length $5n$.
Applying the hook-length bound to the staircase partition
$(n,n-1,\ldots,1)$ gives a filtration count growing exponentially in
the square of this length.

\begin{theorem}\label{thm:growth}
Let $k$ be algebraically closed and $A=kQ_3$. There is a constant
$C>0$ such that, for every $n\geq1$,
a nonempty Zariski-open subset of $\Rep(Q_3,(3n,2n))$ consists of bricks
$M$ satisfying
\[
 \log\bcount_A(M)\geq(1-\log2)n^2-Cn.
\]
\end{theorem}

\begin{proof}
Apply Theorem~\ref{thm:count}\ref{item:filtration-count} with $q=3$ and $a=b=n$. The staircase
$\delta_n=(n,n-1,\ldots,1)$ lies in $\parts_{n,n}$ and has size
$d_n=n(n+1)/2$. At its cell $(i,j)$ the hook length is
$2(n+1-i-j)+1$. For $0\leq t\leq n-1$, the diagonal
$i+j=n+1-t$ has $n-t$ cells, all with hook length $2t+1$. Hence
\[
 H_n:=H(\delta_n)=\prod_{j=0}^{n-1}(2j+1)^{n-j},
 \qquad
 K(3,n,n)\geq L_n:=\frac{2^{d_n}d_n!}{H_n^2}.
\]
We claim that
\begin{equation}\label{eq:staircase-asymptotic}
 \log L_n=(1-\log2)n^2+O(n).
\end{equation}
Set $d_0=0$ and $H_0=L_0=1$. Since
$H_n/H_{n-1}=(2n-1)!!$, the exact ratio is
\begin{equation}\label{eq:staircase-ratio}
 \frac{L_n}{L_{n-1}}
   =\frac{2^n\prod_{j=1}^n(d_{n-1}+j)}{((2n-1)!!)^2}.
\end{equation}
For $n\geq2$, all factors $d_{n-1}+j$ equal
$(n^2/2)(1+O(1/n))$, uniformly for $1\leq j\leq n$. Therefore
\[
 \log\prod_{j=1}^n(d_{n-1}+j)
     =n(2\log n-\log2)+O(1).
\]
Use Stirling's formula with its logarithmic term:
\[
 \log(n!)=(n+\tfrac12)\log n-n+\tfrac12\log(2\pi)+O(n^{-1}).
\]
Since $(2n-1)!!=(2n)!/(2^n n!)$, subtraction gives
\begin{align*}
 \log((2n-1)!!)
 &=\log((2n)!)-n\log2-\log(n!)\\
 &=n\log n+n\log2-n+\tfrac12\log2+O(n^{-1}).
\end{align*}
Substituting in \eqref{eq:staircase-ratio} yields
$\log(L_n/L_{n-1})=2(1-\log2)n+O(1)$.
Summation proves \eqref{eq:staircase-asymptotic}, and
$\bcount_A(M)\geq K(3,n,n)\geq L_n$ proves the theorem.
\end{proof}

For these modules $m=\dim_k M=\ell(M)=5n$, so the lower bound becomes
\[
 \log\bcount_A(M)\geq\frac{1-\log2}{25}m^2-C'm
\]
for a constant $C'>0$. Together with Theorem~\ref{thm:uniform} and
Corollary~\ref{cor:uniform-length}, this shows that the quadratic order
of the exponent is necessary in both dimension and composition length.
Also, $\log(m!)=O(m\log m)$, so these modules have more than $m!$
filtrations for all sufficiently large $n$. This gives a family
answering Question~\ref{ques:ringel}(2) affirmatively.

\begin{remark}
Fix an algebraically closed field $k$, and let $F(m)$ be the supremum
of $\log\bcount_A(M)$ over finite-dimensional $k$-algebras $A$ and
$A$-modules $M$ of composition length $m$. Every such algebra is split,
so Corollary~\ref{cor:uniform-length} and Theorem~\ref{thm:growth} show
that
\[
 \frac{1-\log2}{25}
 \leq \limsup_{m\to\infty}\frac{F(m)}{m^2}
 \leq \log2.
\]
Determining the value of this limit superior remains open.
\end{remark}

\subsection{A module of length 244 with more than \texorpdfstring{$244!$}{244!} filtrations}

The modules in Theorem~\ref{thm:count} have length $2a+qb$.
We choose $q,a,b$ to obtain a module of length $244$ with more than
$244!$ brick chain filtrations.
For a partition $\nu\subseteq b^a$ of size $d$, these modules have
more than $(2a+qb)!$ brick chain filtrations whenever
\[
 \frac{(q-1)^d d!}{H(\nu)^2}>(2a+qb)!.
\]
Increasing $q$ increases both the factor $(q-1)^d$ in the lower bound
and the module length $2a+qb$. The choice $q=5$, $a=47$, $b=30$
and the partition below satisfies this inequality, giving the stated
length. Appendix~\ref{sec:certificate} verifies the inequality by
exact integer arithmetic.
\begin{corollary}\label{cor:244}
A nonempty Zariski-open subset of $\Rep(Q_5,(167,77))$ consists of
bricks $M$ with more than $16\cdot244!$ two-step brick chain
filtrations. These modules have composition length $244$.
\end{corollary}

\begin{proof}
In Theorem~\ref{thm:count}\ref{item:filtration-count}, take $q=5$, $a=47$ and $b=30$.
Then $\beta=(47,47)$, $\gamma=(120,30)$ and $\alpha=(167,77)$.
Let
\[
 \nu=(\underbrace{30,\ldots,30}_{24\text{ parts}},29,28,\ldots,7),
 \qquad |\nu|=1134.
\]
This is the $47$-row, $30$-column rectangle with a staircase of
$23$ rows removed, so $|\nu|=47\cdot30-23\cdot24/2$.
Proposition~\ref{prop:hook-bound} gives
\[
 K(5,47,30)\geq\frac{4^{1134}\,1134!}{H(\nu)^2}
                   >16\cdot244!.
\]
The strict inequality is the integer calculation in
Appendix~\ref{sec:certificate}. Theorem~\ref{thm:count}\ref{item:filtration-count} supplies the
required open set, and $167+77=244$ gives the composition length.
\end{proof}

Theorem~\ref{thm:count}\ref{item:distinct-types} shows that the types of the counted filtrations
are distinct. Thus this example also has more than $16\cdot244!$ types
of brick chain filtrations.

\appendix
\section{Verification of the length-244 inequality}\label{sec:certificate}

We verify the integer inequality used in the proof of
Corollary~\ref{cor:244}. To check the hook product entering that
inequality, we compute it both from the individual hook lengths and
from the following product formula.
For a partition $\nu$ with $r$ parts, put
$l_i=\nu_i+r-i$. Then
\begin{equation}\label{eq:hook-vandermonde}
 H(\nu)=\frac{\prod_{i=1}^r l_i!}
               {\prod_{1\leq i<j\leq r}(l_i-l_j)}.
\end{equation}
\begin{proof}[Proof of \eqref{eq:hook-vandermonde}]
For a cell $(i,j)$, put $t_j=r+j-1-\nu'_j$. Then
$h_\nu(i,j)=l_i-t_j$. For fixed $i$, we will partition
$\{0,\ldots,l_i-1\}$ into the values $t_j$ for cells in row $i$
and the values $l_u$ for rows $u>i$.
The integers $t_j$ with
$1\leq j\leq\nu_i$ are strictly increasing because
$t_{j+1}-t_j=1+\nu'_j-\nu'_{j+1}\geq1$. Their endpoints satisfy
\[
 t_1=r-\nu'_1\geq0,\qquad
 t_{\nu_i}=r+\nu_i-1-\nu'_{\nu_i}
   \leq r+\nu_i-1-i=l_i-1.
\]
The sequence $(l_u)_u$ is strictly decreasing, since
$l_u-l_{u+1}=\nu_u-\nu_{u+1}+1\geq1$, and its terms are nonnegative.
Thus the $r-i$ integers $l_u$ with $u>i$ also lie in $0,\ldots,l_i-1$.
They are disjoint from the $t_j$: if $\nu_u\geq j$, then
$u\leq\nu'_j$ and $l_u>t_j$; if $\nu_u<j$, then
$u>\nu'_j$ and $l_u<t_j$.
Their total number is $\nu_i+r-i=l_i$, so together they fill this
range. Subtracting the integers $t_j$ and $l_u$ from $l_i$ shows that
the hooks in row $i$
are $1,\ldots,l_i$ with the values $l_i-l_u$ for $u>i$ removed.
Multiplying their product over all rows proves
\eqref{eq:hook-vandermonde}.
\end{proof}

For the partition used in the proof of Corollary~\ref{cor:244},
the set of $l_i$ is
\[
 \{53,54,\ldots,76\}\ \cup\ \{7,9,\ldots,51\}.
\]
Computing $H(\nu)$ either directly from its cell hooks or from
\eqref{eq:hook-vandermonde} gives the same integer. The agreement is an
internal check on both the partition data and the product calculation.
Integer division gives
\begin{equation}\label{eq:certificate}
 16<\frac{4^{1134}\,1134!}{244!\,H(\nu)^2}<17.
\end{equation}

The following program, for Python 3.8 or later, verifies both
expressions for $H(\nu)$ and
\eqref{eq:certificate}. It uses exact integer arithmetic throughout.

\begingroup\small
\begin{verbatim}
from math import factorial, prod
import json

nu = [30] * 24 + list(range(29, 6, -1))
assert len(nu) == 47 and sum(nu) == 1134
columns = [sum(row >= j for row in nu) for j in range(1, 31)]
hooks = [nu[i] - j + columns[j] - i - 1
         for i in range(47) for j in range(nu[i])]
hook_product = prod(hooks)

shifted = [nu[i] + 46 - i for i in range(47)]
assert set(shifted) == set(range(53, 77)) | set(range(7, 52, 2))
vandermonde = prod(shifted[i] - shifted[j]
                  for i in range(47) for j in range(i + 1, 47))
assert hook_product * vandermonde == prod(factorial(x) for x in shifted)

numerator = 4 ** 1134 * factorial(1134)
denominator = factorial(244) * hook_product ** 2
quotient, remainder = divmod(numerator, denominator)
strict_bounds = 16 * denominator < numerator < 17 * denominator
assert quotient == 16 and 0 < remainder < denominator and strict_bounds
print(json.dumps({"partition_size": 1134, "module_length": 244,
                  "floor_of_ratio": quotient, "strict_bounds": strict_bounds}))
\end{verbatim}
\endgroup

\end{document}